\documentclass[11pt]{amsart}
\usepackage{mathrsfs,latexsym,amsfonts,amssymb}
\newtheorem{theorem}{Theorem}[section]
\newtheorem{lemma}[theorem]{Lemma}
\newtheorem{corollary}[theorem]{Corollary}
\newtheorem{question}[theorem]{Question}
\newtheorem{example}[theorem]{Example}
\theoremstyle{definition}
\newtheorem{definition}[theorem]{Definition}
\newtheorem{proposition}[theorem]{Proposition}
\theoremstyle{remark}

\begin{document}

\title[Separability in (strongly) topological gyrogroups]
{Separability in (strongly) topological gyrogroups}

\author{Meng Bao}
\address{(Meng Bao): College of Mathematics, Sichuan University, Chengdu 610064, P. R. China}
\email{mengbao95213@163.com}

\author{Xiaoyuan Zhang}
\address{(Xiaoyuan Zhang): 1. College of Mathematics, Sichuan University, Chengdu 610064, P. R. China; 2. School of Big Data Science, Hebei Finance University, Baoding 071051, P. R. China}
\email{405518791@qq.com}

\author{Xiaoquan Xu*}
\address{(Xiaoquan Xu): School of mathematics and statistics,
Minnan Normal University, Zhangzhou 363000, P. R. China}
\email{xiqxu2002@163.com}

\thanks{The authors are supported by the National Natural Science Foundation of China (Nos. 11661057, 12071199) and the Natural Science Foundation of Jiangxi Province, China (No. 20192ACBL20045)\\
*corresponding author}

\keywords{topological gyrogroups; strongly topological gyrogroups; left $\omega$-narrow; feathered; separability.}
\subjclass[2010]{Primary 54A20; secondary 11B05; 26A03; 40A05; 40A30; 40A99.}

\begin{abstract}
Separability is one of the most basic and important topological properties. In this paper, the separability in (strongly) topological gyrogroups is studied. It is proved that every first-countable left $\omega$-narrow strongly topological gyrogroup is separable. Furthermore, it is shown that if a feathered strongly topological gyrogroup $G$ is isomorphic to a subgyrogroup of a separable strongly topological gyrogroup, then $G$ is separable. Therefore, if a metrizable strongly topological gyrogroup $G$ is isomorphic to a subgyrogroup of a separable strongly topological gyrogroup, then $G$ is separable, and if a locally compact strongly topological gyrogroup $G$ is isomorphic to a subgyrogroup of a separable strongly topological gyrogroup, then $G$ is separable.
\end{abstract}

\maketitle
\section{Introduction}
In 2002, A.A. Ungar studied the $c$-ball of relativistically admissible velocities with Einstein velocity addition in \cite{UA2002} and he posed the concept of a gyrogroup. As we all know, the Einstein velocity addition $\oplus _{E}$ is given as the following: $$\mathbf{u}\oplus _{E}\mathbf{v}=\frac{1}{1+\frac{\mathbf{u}\cdot \mathbf{v}}{c^{2}}}(\mathbf{u}+\frac{1}{\gamma _{\mathbf{u}}}\mathbf{v}+\frac{1}{c^{2}}\frac{\gamma _{\mathbf{u}}}{1+\gamma _{\mathbf{u}}}(\mathbf{u}\cdot \mathbf{v})\mathbf{u}),$$ where $\mathbf{u,v}\in \mathbb{R}_{c}^{3}=\{\mathbf{v}\in \mathbb{R}^{3}:||\mathbf{v}||<c\}$ and $\gamma _{\mathbf{u}}$ is given by $$\gamma _{\mathbf{u}}=\frac{1}{\sqrt{1-\frac{\mathbf{u}\cdot \mathbf{u}}{c^{2}}}}.$$ It is well-known that the gyrogroup has a weaker algebraic structure than a group. In 2017, W. Atiponrat \cite{AW} introduced the topological gyrogroups. A gyrogroup $G$ is endowed with a topology such that the binary operation $\oplus: G\times G\rightarrow G$ is jointly continuous and the inverse mapping $\ominus (\cdot): G\rightarrow G$, i.e. $x\rightarrow \ominus x$, is also continuous. He claimed that $T_{0}$ and $T_{3}$ are equivalent with each other in topological gyrogroups. Then Z. Cai, S. Lin and W. He in \cite{CZ} proved that every topological gyrogroup is a rectifiable space. In 2019, the authors \cite{BL} defined the concept of strongly topological gyrogroups and found that M\"{o}bius gyrogroups, Einstein gyrogroups, and Proper velocity gyrogroups are all strongly topological gyrogroups. Furthermore, the authors gave a characterization for a strongly topological gyrogroup being a feathered space, that is, a strongly topological gyrogroup $G$ is feathered if and only if it contains a compact $L$-subgyrogroup $H$ such that the quotient space $G/H$ is metrizable. Therefore, the authors proved that every feathered strongly topological gyrogroup is paracompact. Moreover, the authors proved that every strongly topological gyrogroup with a countable pseudocharacter is submetrizable and every locally paracompact strongly topological gyrogroup is paracompact, see \cite{BL1,BL2}.

A topological gyrogroup $G$ is called left (right) $\omega$-narrow if, for every open neighborhood $V$ of the identity element $0$ in $G$, there exists a countable subset $A$ of $G$ such that $G=A\oplus V$ $(G=V\oplus A)$. If $G$ is left $\omega$-narrow and right $\omega$-narrow, then $G$ is $\omega$-narrow. Moreover, a topological gyrogroup $G$ is feathered if it contains a non-empty compact subset $K$ of countable character in $G$. In Section 3, we show that the topological product of an arbitrary family of $\omega$-narrow topological gyrogroups is an $\omega$-narrow topological gyrogroup and the product of countably many feathered topological gyrogroups is a feathered topological gyrogroup. Furthermore, we also investigate the relationship between the property of (left) $\omega$-narrow and separability in strongly topological gyrogroups, and prove that every first-countable left $\omega$-narrow strongly topological gyrogroup is separable, which gives a partial answer to \cite[Question 6.13]{LF4}.

A topological space which has a dense countable subspace is called {\it separable}. It is well-known that a subspace of a separable metrizable space is separable, but a closed subspace of a separable Hausdorff topological space is not necessarily separable \cite{LMT}. Even though $Y$ is a closed linear subspace of a separable Hausdorff topological vector space $X$, $Y$ is not necessarily separable \cite{LW}. Therefore, it is meaningful to study the relative properties in topological gyrogroups or strongly topological gyrogroups. In particular, we want to know under what conditions a (closed) subgyrogroup of a (strongly) topological gyrogroup is separable. In Section 4, we show that if a feathered strongly topological gyrogroup $G$ is isomorphic to a subgyrogroup of a separable strongly topological gyrogroup, then $G$ is separable. Therefore, we deduce that if a metrizable strongly topological gyrogroup $G$ is isomorphic to a subgyrogroup of a separable strongly topological gyrogroup, then $G$ is separable, and if a locally compact strongly topological gyrogroup $G$ is isomorphic to a subgyrogroup of a separable strongly topological gyrogroup, then $G$ is separable.

\smallskip
\section{Preliminaries}
In this section, we introduce the necessary notations, terminologies and some facts about topological gyrogroups.

Throughout this paper, all topological spaces are assumed to be Hausdorff, unless otherwise is explicitly stated. Let $\mathbb{N}$ be the set of all positive integers and $\omega$ the first infinite ordinal. Let $X$ be a topological space and $A \subseteq X$ be a subset of $X$.
  The {\it closure} of $A$ in $X$ is denoted by $\overline{A}$ and the
  {\it interior} of $A$ in $X$ is denoted by $\mbox{Int}(A)$. The readers may consult \cite{AA, E, linbook} for notation and terminology not explicitly given here.
\begin{definition}\cite{AW}
Let $G$ be a nonempty set, and let $\oplus: G\times G\rightarrow G$ be a binary operation on $G$. Then the pair $(G, \oplus)$ is called a {\it groupoid}. A function $f$ from a groupoid $(G_{1}, \oplus_{1})$ to a groupoid $(G_{2}, \oplus_{2})$ is called a {\it groupoid homomorphism} if $f(x\oplus_{1}y)=f(x)\oplus_{2} f(y)$ for any elements $x, y\in G_{1}$. Furthermore, a bijective groupoid homomorphism from a groupoid $(G, \oplus)$ to itself will be called a {\it groupoid automorphism}. We write $\mbox{Aut}(G, \oplus)$ for the set of all automorphisms of a groupoid $(G, \oplus)$.
\end{definition}

\begin{definition}\cite{UA}
Let $(G, \oplus)$ be a groupoid. The system $(G,\oplus)$ is called a {\it gyrogroup}, if its binary operation satisfies the following conditions:

\smallskip
(G1) There exists a unique identity element $0\in G$ such that $0\oplus a=a=a\oplus0$ for all $a\in G$.

\smallskip
(G2) For each $x\in G$, there exists a unique inverse element $\ominus x\in G$ such that $\ominus x \oplus x=0=x\oplus (\ominus x)$.

\smallskip
(G3) For all $x, y\in G$, there exists $\mbox{gyr}[x, y]\in \mbox{Aut}(G, \oplus)$ with the property that $x\oplus (y\oplus z)=(x\oplus y)\oplus \mbox{gyr}[x, y](z)$ for all $z\in G$.

\smallskip
(G4) For any $x, y\in G$, $\mbox{gyr}[x\oplus y, y]=\mbox{gyr}[x, y]$.
\end{definition}

Notice that a group is a gyrogroup $(G,\oplus)$ such that $\mbox{gyr}[x,y]$ is the identity function for all $x, y\in G$. The definition of a subgyrogroup is given as follows.

\begin{definition}\cite{ST}
Let $(G,\oplus)$ be a gyrogroup. A nonempty subset $H$ of $G$ is called a {\it subgyrogroup}, denoted
by $H\leq G$, if $H$ forms a gyrogroup under the operation inherited from $G$ and the restriction of $gyr[a,b]$ to $H$ is an automorphism of $H$ for all $a,b\in H$.

\smallskip
Furthermore, a subgyrogroup $H$ of $G$ is said to be an {\it $L$-subgyrogroup}, denoted
by $H\leq_{L} G$, if $gyr[a, h](H)=H$ for all $a\in G$ and $h\in H$.
\end{definition}

\begin{definition}\cite{AW}
A triple $(G, \tau, \oplus)$ is called a {\it topological gyrogroup} if the following statements hold:

\smallskip
(1) $(G, \tau)$ is a topological space.

\smallskip
(2) $(G, \oplus)$ is a gyrogroup.

\smallskip
(3) The binary operation $\oplus: G\times G\rightarrow G$ is jointly continuous while $G\times G$ is endowed with the product topology, and the operation of taking the inverse $\ominus (\cdot): G\rightarrow G$, i.e. $x\rightarrow \ominus x$, is also continuous.
\end{definition}

Obviously, every topological group is a topological gyrogroup. However, every topological gyrogroup whose gyrations are not identically equal to the identity is not a topological group.

\begin{example}\cite{AW}
The Einstein gyrogroup with the standard topology is a topological gyrogroup but not a topological group.
\end{example}

The Einstein gyrogroup has been introduced in the Introduction. It was proved in \cite{UA} that $(R^{3}_{c},\oplus _{E})$ is a gyrogroup but not a group. Moreover, with the standard topology inherited from $R^{3}$, it is clear that $\oplus _{E}$ is continuous. Finally, $-\mathbf{u}$ is the inverse of $\mathbf{u}\in R^{3}$ and the operation of taking the inverse is also continuous. Therefore, the Einstein gyrogroup $(R^{3}_{c},\oplus _{E})$ with the standard topology inherited from $R^{3}$ is a topological gyrogroup but not a topological group.

Next, we introduce the definition of a strongly topological gyrogroup, it is very important in this paper.

\begin{definition}{\rm (\cite{BL})}\label{d11}
Let $G$ be a topological gyrogroup. We say that $G$ is a {\it strongly topological gyrogroup} if there exists a neighborhood base $\mathscr U$ of $0$ such that, for every $U\in \mathscr U$, $\mbox{gyr}[x, y](U)=U$ for any $x, y\in G$. For convenience, we say that $G$ is a strongly topological gyrogroup with neighborhood base $\mathscr U$ of $0$.
\end{definition}

For each $U\in \mathscr U$, we can set $V=U\cup (\ominus U)$. Then, $$gyr[x,y](V)=gyr[x, y](U\cup (\ominus U))=gyr[x, y](U)\cup (\ominus gyr[x, y](U))=U\cup (\ominus U)=V,$$ for all $x, y\in G$. Obviously, the family $\{U\cup(\ominus U): U\in \mathscr U\}$ is also a neighborhood base of $0$. Therefore, we may assume that $U$ is symmetric for each $U\in\mathscr U$ in Definition~\ref{d11}. Moreover, it is clear that every topological group is a strongly topological gyrogroup, and every strongly topological gyrogroup is a topological gyrogroup. However, there is a strongly topological gyrogroup which is not a topological group, see the following Example \ref{lz1}.

\begin{example}\cite{BL}\label{lz1}
Let $D$ be the complex open unit disk $\{z\in C:|z|<1\}$. We consider $D$ with the standard topology. As in \cite[Example 2]{AW}, a M\"{o}bius addition $\oplus _{M}: D\times D\rightarrow D$ is a function such that $$a\oplus _{M}b=\frac{a+b}{1+\bar{a}b}\ \mbox{for all}\ a, b\in D.$$ Then $(D, \oplus _{M})$ is a gyrogroup, and it follows from \cite[Example 2]{AW} that $$gyr[a, b](c)=\frac{1+a\bar{b}}{1+\bar{a}b}c\ \mbox{for any}\ a, b, c\in D.$$ For any $n\in\omega$, let $U_{n}=\{x\in D: |x|\leq \frac{1}{n}\}$. Then, $\mathscr U=\{U_{n}: n\in \omega\}$ is a neighborhood base of $0$. Moreover, we observe that $|\frac{1+a\bar{b}}{1+\bar{a}b}|=1$. Therefore, we obtain that $gyr[x, y](U)\subset U$, for any $x, y\in D$ and each $U\in \mathscr U$, then it follows that $gyr[x, y](U)=U$ by \cite[Proposition 2.6]{ST}. Hence, $(D, \oplus _{M})$ is a strongly topological gyrogroup. However, $(D, \oplus _{M})$ is not a group \cite[Example 2]{AW}.
\end{example}

Indeed, we know that M\"{o}bius gyrogroups, Einstein gyrogroups, and Proper velocity gyrogroups, that were studied in \cite{FM, FM1,UA}, are all strongly topological gyrogroups. Therefore, they are all topological gyrogroups and rectifiable spaces. But all of them are not topological groups. Further, it was also proved in \cite[Example 3.2]{BL} that there exists a strongly topological gyrogroup which has an infinite $L$-subgyrogroup.

Moreover, in \cite{BL1}, the authors proved that every $T_{0}$ strongly topological gyrogroup is completely regular. Then, we will give an example to show that there is a completely regular strongly topological gyrogroup which is not a normal space.

\begin{example}
There is a completely regular strongly topological gyrogroup which is not a normal space.
\end{example}

Indeed, let $X$ be an arbitrary $T_{0}$ strongly topological gyrogroup (such us Example \ref{lz1}), and let $Y$ be an any $T_{0}$ topological group. Put $G=X\times Y$ with the product topology and the operation with coordinate. Then $G$ is an completely regular strongly topological gyrogroup since $X$ and $Y$ both are completely regular. However, $G$ is not a normal space.

\section{Products of two classes of topological gyrogroups}
In this section, we mainly study the products of $\omega$-narrow topological gyrogroups and feathered topological gyrogroups. We show that the topological product of an arbitrary family of $\omega$-narrow topological gyrogroups is an $\omega$-narrow topological gyrogroup and the product of countably many feathered topological gyrogroups is a feathered topological gyrogroup. Moreover, we prove that every first-countable left $\omega$-narrow strongly topological gyrogroup is separable, which gives a partial answer to \cite[Question 6.13]{LF4}.

A topological gyrogroup $G$ is called left (right) $\omega$-narrow \cite{AA} if, for every open neighborhood $V$ of the identity element $0$ in $G$, there exists a countable subset $A$ of $G$ such that $G=A\oplus V$ $(G=V\oplus A)$. If $G$ is left $\omega$-narrow and right $\omega$-narrow, then $G$ is $\omega$-narrow. Moreover, it was proved that the quotient space $G/H$ is homogeneous in \cite{BL1} if $G$ is a strongly topological gyrogroup with a symmetric neighborhood base $\mathscr U$ and $H$ is an $L$-subgyrogroup generated from $\mathscr U$. Therefore, the definition of $\omega$-narrow of $G/H$ is the same with gyrogroups'.

\begin{proposition}\label{mt1}
If a topological gyrogroup $H$ is a continuous homomorphic image of an $\omega$-narrow topological gyrogroup $G$, then $H$ is also $\omega$-narrow.
\end{proposition}

\begin{proof}
We just prove the situation of left $\omega$-narrow. For the situation of right $\omega$-narrow, the proof is similar. For an arbitrary open neighborhood $V$ of the identity element $0$ in $H$, since $f$ is a continuous homomorphism from $G$ onto $H$, it follows that $f^{-1}(V)$ is an open neighborhood of $0$ in $G$. It follows from the left $\omega$-narrow property of $G$ that there exists a countable subset $A$ of $G$ such that $G=A\oplus f^{-1}(V)$. Therefore, $$H=f(G)=f(A\oplus f^{-1}(V))=f(A)\oplus f(f^{-1}(V))\subset f(A)\oplus V.$$ Since $f(A)$ is countable, it is clear that $H$ is left $\omega$-narrow.
\end{proof}

\begin{proposition}\label{3mt2}
The topological product of an arbitrary family of $\omega$-narrow topological gyrogroups is an $\omega$-narrow topological gyrogroup.
\end{proposition}

\begin{proof}
We also just prove the situation of left $\omega$-narrow. For the situation of right $\omega$-narrow, the proof is similar. Let $\{(G_{i},\tau _{i},\oplus _{i}):i\in I\}$ be an indexed family of left $\omega$-narrow topological gyrogroups. It follows from Theorem 2.1 in \cite{ST2} and Theorem 5 in \cite{AW} that $G=(\prod _{i\in I}G_{i},\oplus)$ is a topological gyrogroup equipped with the product topology. Then we show that $(\prod _{i\in I}G_{i},\oplus)$ is left $\omega$-narrow. Let $U$ be a basic open subset of $\prod _{i\in I}G_{i}$. Then $U=\prod _{i\in I}U_{i}$, where $U_{i}$ is open in $G_{i}$ for each $i\in I$. For the product topology, we know that $U_{i}\not =G_{i}$ for only finitely many $i\in I$. Therefore let $U=\prod _{i\in J}U_{i}\times \prod _{i\in I\setminus J}G_{i}$, where $J$ is a finite subset of $I$. Since every $G_{i}$ is left $\omega$-narrow, it follows that there exists a countable subset $A_{i}$ of $G_{i}$ such that $G_{i}=A_{i}\oplus U_{i}$. Set $A=\bigcup _{i\in J}A_{i}$. It is obvious that $A$ is countable. Moreover, $G\subset A\oplus U$ and the proof is completed.
\end{proof}

\begin{theorem}
Suppose that $(G,\tau ,\oplus)$ is a strongly topological gyrogroup with a synmetric neighborhood base $\mathscr U$ at $0$. If $G$ contains a dense subgyrogroup $H$ such that $H$ is left $\omega$-narrow, then $G$ is also left $\omega$-narrow.
\end{theorem}

\begin{proof}
If $U$ is an open neighborhood of $0$ in $G$, we can choose $V\in \mathscr U$ such that $V\oplus V\subset U$. Since $H$ is left $\omega$-narrow, there is a countable subset $A$ of $H$ such that $H\subset A\oplus V$. Therefore, by \cite[Lemma 9]{AW}, we have $$G=\overline{H} \subset (A\oplus V)\oplus V=A\oplus (V\oplus gyr[V,A](V))=A\oplus (V\oplus V)\subset A\oplus U.$$ Therefore, $G=A\oplus U$ and $G$ is left $\omega$-narrow.
\end{proof}

In \cite{BL3}, the authors gave the following result.

\begin{proposition}\label{ktl1}\cite{BL3}
Every separable strongly topological gyrogroup $G$ is left $\omega$-narrow.
\end{proposition}

Then, we will show that the left $\omega$-narrow strongly topological gyrogroups need not be separable, that is, there exists a left $\omega$-narrow strongly topological gyrogroup which is not separable.

\begin{example}
There exists a left $\omega$-narrow strongly topological gyrogroup which is not separable.
\end{example}

Let $X$ be an arbitrary left $\omega$-narrow strongly topological gyrogroup, and let $Y$ be an $\omega$-narrow topological group which has uncountable cellularity, i.e., there is an uncountable family of disjoint non-empty open subsets in $Y$ (such as \cite[Example 5.4.13]{AA}). Set $G=X\times Y$ with the product topology and the operation with coordinate. Then $G$ is a left $\omega$-narrow strongly topological gyrogroup by Proposition \ref{3mt2} and there is an uncountable family of disjoint non-empty open subsets in $G$.

\bigskip
Next, a family $\gamma$ of open sets in a space $X$ is called a {\it base} \cite{E} for $X$ at a set $F\subset X$ if all elements of $\gamma$ contains $F$ and, for each open set $V$ that contains $F$, there exists $U\in\gamma$ such that $U\subset V$. The {\it character} \cite{E} of $X$ at a set $F$ is the smallest cardinality of a base for $X$ at $F$. We recall the definition of the {\it feathered (strongly) topological gyrogroup}. A (strongly) topological gyrogroup $G$ is {\it feathered} if it contains a non-empty compact subset $K$ of countable character in $G$. In \cite{BL}, it was proved that a strongly topological gyrogroup $G$ is feathered if and only if it contains a compact $L$-subgyrogroup $H$ such that the quotient space $G/H$ is metrizable. Moreover, it was also proved in the same paper that every feathered strongly topological gyrogroup is paracompact and every feathered strongly topological gyrogroup is a $D$-space. Then, we will prove that the class of feathered strongly topological gyrogroups is closed under taking countable products. Moreover, we will give an example to show that the product of arbitrary family of feathered topological gyrogroups need not to be feathered.

\begin{lemma}\cite{E}\label{wdl}
If $A_{s}$ is a compact subspace of a topological space $X_{s}$ for $s\in S$, then for every open subset $W$ of the Cartesian product $\prod _{s\in S}X_{s}$ which contains the set $\prod _{s\in S}A_{s}$ there exist open sets $U_{s}\subset X_{s}$ such that $U_{s}\not =X_{s}$ only for finitely many $s\in S$ and $\prod _{s\in S}A_{s}\subset \prod _{s\in S}U_{s}\subset W$.
\end{lemma}

\begin{theorem}\label{dlc2}
The product $G=\prod _{n\in \omega}G_{n}$ of countably many feathered topological gyrogroups is a feathered topological gyrogroup.
\end{theorem}

\begin{proof}
Let $K_{n}$ be the non-empty compact subset of countable character of $G_{n}$ containing the identity element $0$ for every $n\in \omega$. Set $K=\prod _{n\in \omega}K_{n}$. It is clear that $K$ is compact by the Tychonoff Product Theorem. Let $\gamma _{n}$ be a countable base for $G_{n}$ at $K_{n}$, $n\in \omega$. We show that $$\mathscr B=\{\pi _{0}^{-1}(U_{0})\cap \ldots \cap \pi _{k}^{-1}(U_{k}):U_{0}\in \gamma _{0},\ldots ,U_{k}\in \gamma _{k},k\in \omega\}$$ is a base for $G$ at $K$, where $\pi _{i}:G\rightarrow G_{i}$ is the projection for each $i\in \omega$.

In fact, let $W$ be a neighborhood of $K$ in $G$. It follows from Lemma \ref{wdl} that there exist open sets $W_{n}\subset G_{n}$ such that $W_{n}\not =G_{n}$ for only finitely many $n\in \omega$ and $K\subset \prod _{n\in \omega}W_{n}\subset W$. We can find $k\in \omega$ such that $W_{n}=G_{n}$ for all $n>k$ and, for every $i\leq k$,
choose $U_{i}\in \gamma _{i}$ satisfying $U_{i}\subset W_{i}$. Then, $U=\pi _{0}^{-1}(U_{0})\cap \ldots \cap \pi _{k}^{-1}(U_{k})$ belongs to $\mathscr B$. Moreover, $K\subset U\subset W$. Hence, we have that $\chi (K,G)\leq |\mathscr B|\leq \omega$.
\end{proof}

\begin{theorem}
If a topological gyrogroup $H$ is a continuous homomorphic image of a feathered topological gyrogroup $G$, then $H$ is also feathered.
\end{theorem}

\begin{proof}
We assume that $f$ is a continuous homomorphism from a feathered topological gyrogroup $G$ onto a topological gyrogroup $H$. Since $G$ is feathered, there is a non-empty compact set $K$ of countable character contained in $G$. Let $\{U_{n}:n\in \omega\}$ be a countable base for $G$ at $K$. It is clear that $f(K)$ is a non-empty compact subset of $H$. We show that $f(K)$ is of countable character in $H$.

Suppose that $V$ is an arbitrary open neighborhood of $f(K)$ in $H$, then $f^{-1}(V)$ is an open neighborhood of $K$ in $G$. Moreover, $\{U_{n}:n\in \omega\}$ is a countable base at $K$, so there exists $n\in \omega$ such that $K\subset U_{n}\subset f^{-1}(V)$. Therefore, we have $f(K)\subset f(U_{n})\subset V$. Thus, $\{f(U_{n}):n\in \omega\}$ is a countable base at $f(K)$ in $H$ and $H$ is feathered.
\end{proof}

Since projection is an open continuous homomorphism, it is natural to have the following corollary.

\begin{corollary}\label{tl2}
If $G$ is a feathered topological gyrogroup and $G=\prod _{n\in \omega}G_{n}$, where $G_{n}$ is a topological gyrogroup for any $n\in \omega$, then $G_{n}$ is feathered for each $n\in \omega$.
\end{corollary}

Then, we will show that the product of arbitrary family of feathered topological gyrogroups need not to be feathered.

\begin{example}
There is a topological gyrogroup which is the product of uncountable many feathered topological gyrogroups but not feathered.
\end{example}

Let $G=H\times Z^{m}$, where $H$ is a feathered strongly topological gyrogroup, $m$ is any uncountable cardinal number. It is clear that every locally compact strongly topological gyrogroup is feathered. Suppose on the contrary, if $G$ is feathered, it follows from Corollary \ref{tl2} that $Z^{m}$ is feathered. In \cite{BL}, it was proved that every feathered strongly topological gyrogroup is paracompact. However, $Z^{m}$ is a non-normal completely regular topological group (see \cite[Theorem 8.11]{EK}) which is contradict with the paracompactness.

\bigskip
In \cite{LF4}, F. Lin posed the following question.

\begin{question}\cite[Question 6.13]{LF4}
Is each first-countable left $\omega$-narrow rectifiable space $G$ separable?
\end{question}

It is well-known that every topological gyrogroup is rectifiable, so it is natural to pose the next question.

\begin{question}\label{wt1}
Is each first-countable left $\omega$-narrow topological gyrogroup $G$ separable? What if the topological gyrogroup is a strongly topological gyrogroup?
\end{question}

Next we prove that every first-countable left $\omega$-narrow strongly topological gyrogroup is separable, which gives an affirmative answer to Question \ref{wt1} when the topological gyrogroup is a strongly topological gyrogroup, see Corollary \ref{wtl1}.

\begin{theorem}\label{kdl3}
Let $(G,\tau ,\oplus)$ be a left $\omega$-narrow strongly topological gyrogroup with a symmetric open neighborhood base $\mathscr U$ at $0$. If $G$ is first-countable, then $G$ has a countable base.
\end{theorem}

\begin{proof}
Let $\{U_{n}:n\in \omega\}$ be a countable base at the identity element $0$ of $G$, then there exists a countable base $\{V_{n}:n\in \omega\}$ at $0$ such that $V_{n}\in \mathscr U$. Since $G$ is left $\omega$-narrow, for every $V_{n}$, there exists a countable subset $A_{n}$ of $G$ such that $G=A_{n}\oplus V_{n}$. Set $\mathscr B=\{x\oplus V_{n}:x\in A_{n}\mbox{ and }n\in \omega\}$. Obviously, $\mathscr B$ is countable and we prove that $\mathscr B$ is a base for the gyrogroup $G$.

For an arbitrary open neighborhood $O$ of a point $a\in G$. It is clear that there are $k,l\in \omega$ such that $a\oplus V_{k}\subset O$ and $V_{l}\oplus V_{l}\subset V_{k}$. Therefore, there exists $x\in A_{l}$ such that $a\in x\oplus V_{l}$. Then there is a $y\in V_{l}$ such that $a=x\oplus y$. It follows that
\begin{eqnarray}
x&=&(x\oplus y)\oplus gyr[x,y](\ominus y)\nonumber\\
&=&a\oplus gyr[x,y](\ominus y)\nonumber\\
&\in &a\oplus gyr[x,y](V_{l})\nonumber\\
&=&a\oplus V_{l}.\nonumber
\end{eqnarray}
So, $x\oplus V_{l}\subset (a\oplus V_{l})\oplus V_{l}=a\oplus (V_{l}\oplus gyr[V_{l},a](V_{l}))=a\oplus (V_{l}\oplus V_{l})\subset a\oplus V_{k}\subset O$, that is, $x\oplus V_{l}$ is an open neighborhood of $a$ and $x\oplus V_{l}\subset O$.
\end{proof}

It follows from \cite{CZ} that every topological gyrogroup is first-countable if and only if it is metrizable. Moreover, it is well-known that the separability, the Lindel\"{o}f property and the second-countability are all equivalent with each other in metrizable spaces. Therefore, we have the following corollaries.

\begin{corollary}\label{wtl1}
Every first-countable left $\omega$-narrow strongly topological gyrogroup is separable.
\end{corollary}

\begin{corollary}
Every first-countable left $\omega$-narrow strongly topological gyrogroup is Lindel\"{o}f.
\end{corollary}

\section{Separability of strongly topological gyrogroups}

In this section, we will study some properties about separabilities of strongly topological gyrogroups. In particular, we prove that if a feathered strongly topological gyrogroup is isomorphic to a subgyrogroup of a separable strongly topological gyrogroup, then it is separable. After that, if a metrizable strongly topological gyrogroup $G$ is isomorphic to a subgyrogroup of a strongly topological gyrogroup with countable cellularity, then $G$ is separable. And if a locally compact strongly topological gyrogroup $G$ is isomorphic to a subgyrogroup of a separable strongly topological gyrogroup, then $G$ is separable.

First, we recall the concept of the coset space of a topological gyrogroup.

Let $(G, \tau, \oplus)$ be a topological gyrogroup and $H$ an $L$-subgyrogroup of $G$. It follows from \cite[Theorem 20]{ST} that $G/H=\{a\oplus H:a\in G\}$ is a partition of $G$. We denote by $\pi$ the mapping $a\mapsto a\oplus H$ from $G$ onto $G/H$. Clearly, for each $a\in G$, we have $\pi^{-1}\{\pi(a)\}=a\oplus H$.
Denote by $\tau (G)$ the topology of $G$. In the set $G/H$, we define a family $\tau (G/H)$ of subsets as follows: $$\tau (G/H)=\{O\subset G/H: \pi^{-1}(O)\in \tau (G)\}.$$

\begin{theorem}\label{kdl1}
Suppose that $G$ is a topological gyrogroup and $H$ is a closed $L$-subgyrogroup of $G$. If the spaces $H$ and $G/H$ are both separable, we obtain that the space $G$ is also separable.
\end{theorem}

\begin{proof}
We suppose that $\pi $ is the natural homomorphism of $G$ onto the quotient space $G/H$. From the separability of $G/H$, it follows that there exists a dense countable subset $A$ of $G/H$. Moreover, $H$ is separable and every coset $x\oplus H$ is homeomorphism to $H$, so there is a dense countable subset $M_{y}$ of $\pi ^{-1}(y)$, for each $y\in A$. Set $M=\bigcup \{M_{y}: y\in A\}$. It is obvious that $M$ is a countable subset of $G$ and $M$ is dense in $\pi ^{-1}(A)$. Furthermore, $\pi$ is an open mapping of $G$ onto $G/H$ by \cite[Theorem 3.7]{BL}, and it follows that $\overline{\pi ^{-1}(A)}=G$. Therefore, $M$ is dense in $G$ and $G$ is separable.
\end{proof}

In \cite{CZ}, Z. Cai, S. Lin and W. He proved that every topological gyrogroup is a rectifiable space, which deduced that the first-countability and metrizability are equivalent in topological gyrogroups. Moreover, it is well-known that separability is equivalent with the second-countability in a metrizable space. Therefore, if we can prove that the first-countability has the property like Theorem \ref{kdl1}, it is natural that the second-countability has the same property.

\begin{lemma}\label{3t2}
Suppose that $G$ is a topological gyrogroup, $H$ is a closed $L$-subgyrogroup of $G$, $X$ is a subspace of $G$, $\pi$ is the natural homomorphism of $G$ onto the quotient space $G/H$, and $Y=\pi (X)$. Suppose that the space $H$ and the subspace $Y$ of $G/H$ are first-countable. Then $X$ is also first-countable.
\end{lemma}

\begin{proof}
Without loss of generality, we may assume that $0\in X$. Then we need to verify that $X$ is first-countable at $0$. Take a sequence of symmetric open neighborhoods $W_{n}$ of $0$ in $G$ such that $W_{n+1}\oplus W_{n+1}\subset W_{n}$, for each $n\in \omega$, and $\{W_{n}\cap H:n\in \omega\}$ is a base for the space $H$ at $0$. We also take a sequence of open neighborhoods $U_{n}$ of $0$ in $G$ such that $\{\pi (U_{n})\cap Y:n\in \omega\}$ is a base for $Y$ at $\pi (0)$. Then set $B_{i,j}=W_{i}\cap U_{j}\cap X$, for $i,j\in \omega$.

{\bf Claim:} $\eta =\{B_{i,j}:i,j\in \omega \}$ is a base for $X$ at $0$.

It is obvious that $B_{i,j}$ is open in $X$ and $0\in B_{i,j}$. For an arbitrary open neighborhood $O$ of $0$ in $G$, we can find an open neighborhood $V$ of $0$ in $G$ such that $V\oplus V\subset O$. Fix $m\in \omega$ such that $W_{m}\cap H\subset V$. Moreover, we can find $k\in \omega$ such that $\pi (U_{k})\cap Y\subset \pi (V\cap W_{m+1})$. We show that $B_{m+1,k}\subset O$.

For each $z\in B_{m+1,k}=W_{m+1}\cap U_{k}\cap X$, it follows from $\pi (z)\in \pi (U_{k})\cap Y\subset \pi (V\cap W_{m+1})$ that $z\in U_{k}\cap X\subset (V\cap W_{m+1})\oplus H$. However, $W_{m+1}\oplus W_{m+1}\subset W_{m}$ and $z\in W_{m+1}=(\ominus W_{m+1})$, so $z\not\in W_{m+1}\oplus (G\backslash W_{m})$. Therefore, $z\in (V\cap W_{m+1})\oplus (H\cap W_{m})$. Moreover, $H\cap W_{m}\subset V$, and we have $z\in V\oplus V\subset O$. Hence, $B_{m+1,k}\subset O$ and $\eta$ is a base for $X$ at $0$. It follows from the countability of $\eta$ that $X$ is first-countable at $0$.
\end{proof}

By Lemma \ref{3t2} and Theorem \ref{kdl1}, it is obvious that we have the following results.

\begin{corollary}
Assume that $G$ is a topological gyrogroup and $H$ is a closed $L$-subgyrogroup of $G$. If the space $H$ and $G/H$ are first-countable (metrizable), then the space $G$ is also first-countable (metrizable).
\end{corollary}

\begin{corollary}
Assume that $G$ is a topological gyrogroup, and $H$ is a second-countable closed $L$-subgyrogroup of $G$. If the quotient space $G/H$ is second-countable, then $G$ is also second-countable.
\end{corollary}

A family $\mathcal{N}$ of subsets of a topological space $Y$ is called a {\it network} \cite{E} for $Y$ if for every point $y\in Y$ and any neighborhood $U$ of $y$ there exists a set $F\in \mathcal{N}$ such that $y\in F\subset U$. The {\it network weight nw(Y)} \cite{E} of a space $Y$ is defined as the smallest cardinal number of the form $|\mathcal{N}|$, where $\mathcal{N}$ is a network for $Y$.

Then, we show the main results in this section. First, we need to introduce some lemmas.

\begin{lemma}\cite{LMT}\label{kyl1}
If $L$ is a Lindel\"{o}f subspace of a separable Hausdorff space $X$, then $nw(L)\leq \mathbf{c}$. Hence every compact subspace $K$ of a separable Hausdorff space satisfies $w(K)\leq \mathbf{c}$.
\end{lemma}

\begin{proposition}\label{kmt1}\cite{BL3}
If $(G,\tau ,\oplus)$ is a left $\kappa$-bounded strongly topological gyrogroup with a symmetric open neighborhood base $\mathscr U$ at $0$ and $H$ is a subgyrogroup of $G$, then $H$ is also left $\kappa$-bounded.
\end{proposition}

\begin{lemma}\label{ylw1}
Let $G$ be a topological gyrogroup and $H$ an $L$-subgyrogroup of $G$. If $\varphi$ is a canonical mapping from $G$ onto the quotient space $G/H$ and $G$ is $\omega$-narrow, then $G/H$ is $\omega$-narrow.
\end{lemma}

\begin{proof}
The proof is similar to that of Proposition \ref{mt1}.
\end{proof}

\begin{lemma}\label{kyl2}
Every left $\omega$-narrow feathered strongly topological gyrogroup is Lindel\"{o}f.
\end{lemma}

\begin{proof}
We assume that $G$ is a left $\omega$-narrow feathered strongly topological gyrogroup with a symmetric neighborhood base $\mathscr U$ at the identity element $0$ such that for any $x,y\in G$, $gyr[x,y](U)=U$ for any $U\in \mathscr U$. Since $G$ is feathered, it follows from \cite[Theorem 3.14]{BL} that there exists a compact $L$-subgyrogroup $H$ generated by $\mathscr U$ such that the quotient space $G/H$ is metrizable.

{\bf Claim} The quotient space $G/H$ is Lindel\"{o}f.

Let $\pi :G\rightarrow G/H$ be a natural homomorphism and it follows from \cite[Theorem 3.8]{BL} that $\pi$ is a perfect mapping. Therefore, $G/H$ is left $\omega$-narrow as a continuous homomorphic image of a left $\omega$-narrow topological gyrogroup $G$ by Lemma \ref{ylw1} . Moreover, $G/H$ is metrizable, so it is a first-countable space and we assume that $\mathscr V=\{V_{n}: n\in \omega\}$ is a countable base at the identity of $G/H$. Therefore, $\pi ^{-1}(V_{n})$ is an open neighborhood of $0$ in $G$ for every $n\in \omega$. For each $n\in \omega$, we can find $U_{n}\in \mathscr U$ such that $U_{n}\subset \pi ^{-1}(V_{n})$. Since $G$ is left $\omega$-narrow, there exists a countable set $C_{n}$ such that $G=C_{n}\oplus U_{n}$ for each $n\in \omega$. Set $C=\bigcup _{n\in \omega}C_{n}$. It is clear that $C$ is countable and we show that $\pi (C)$ is dense in the quotient space $G/H$.

For arbitrary open set $W$ in $G/H$, we need to prove $\pi (C)\cap W\not =\emptyset$, that is, $\pi ^{-1}\pi (C)\cap \pi ^{-1}(W)\not =\emptyset$. It means that $(C\oplus H)\cap (W\oplus H)\not =\emptyset$.

{\bf Subclaim} $(C\oplus H)\cap (W\oplus H)\not =\emptyset$ if and only if $C\cap (W\oplus H)\not =\emptyset$.

{\bf Sufficiency:} It is obvious because of $0\in H$.

{\bf Necessity:} If $(C\oplus H)\cap (W\oplus H)\not =\emptyset$, there are $c\in C,w\in W,h_{1},h_{2}\in H$ such that $c\oplus h_{1}=w\oplus h_{2}$. Since $H$ is generated by $\mathscr U$, we have that
\begin{eqnarray}
c&=&(c\oplus h_{1})\oplus gyr[c,h_{1}](\ominus h_{1})\nonumber\\
&=&(w\oplus h_{2})\oplus gyr[c,h_{1}](\ominus h_{1})\nonumber\\
&\in &(w\oplus h_{2})\oplus gyr[c,h_{1}](H)\nonumber\\
&=&(w\oplus h_{2})\oplus H\nonumber\\
&=&w\oplus (h_{2}\oplus gyr[h_{2},w](H))\nonumber\\
&=&w\oplus (h_{2}\oplus H)\nonumber\\
&=&w\oplus H.\nonumber
\end{eqnarray}
Therefore, $C\cap (W\oplus H)\not =\emptyset$.

Thus, it suffices to prove $C\cap (W\oplus H)\not =\emptyset$. Indeed, $W\oplus H$ is open in $G$ and we can find $y\in G$, $U_{n}\in \mathscr U$ and $U_{n}\subset \pi ^{-1}(V_{n})$ for some $n\in \omega$ such that $y\oplus U_{n}\subset W\oplus H$. we show that $C\cap (y\oplus U_{n})\not =\emptyset$. It means that there are $c\in C,u\in U_{n}$ such that $c=y\oplus u$. Therefore, $c\oplus gyr[y,u](\ominus u)=(y\oplus u)\oplus gyr[y,u](\ominus u)=y$. Since $C\oplus U_{n}=G$ and $gyr[y,u](\ominus u)\in gyr[y,u](U_{n})=U_{n}$, it follows that we can find $c$ and $u$ which are satisfied. Therefore, $\pi (C)\cap W\not =\emptyset$ and $G/H$ is separable. Moreover, $G/H$ is metrizable and separability is equivalent with the Lindel\"{o}f property in a metrizable space, so we deduce that the quotient space $G/H$ is Lindel\"{o}f.

Furthermore, the quotient mapping $\pi :G\rightarrow G/H$ is perfect and the property of Lindel\"{o}f is an inverse invariant of perfect mappings, hence we conclude that $G$ is Lindel\"{o}f and we complete the proof.
\end{proof}

\begin{theorem}\label{kdl2}
Let a feathered strongly topological gyrogroup $G$ be isomorphic to a subgyrogroup of a separable strongly topological gyrogroup. Then $G$ is separable.
\end{theorem}

\begin{proof}
Assume that a feathered strongly topological gyrogroup $G$ is a subgyrogroup of a separable strongly topological gyrogroup $X$. It follows from Proposition \ref{ktl1} that the gyrogroup $X$ is left $\omega$-narrow. Hence, according to Proposition \ref{kmt1} that the subgyrogroup $G$ of $X$ is also left $\omega$-narrow. Moreover, every left $\omega$-narrow feathered strongly topological gyrogroup is Lindel\"{o}f by Lemma \ref{kyl2}. Furthermore, it follows from \cite[Theorem 3.14]{BL} that there is a compact $L$-subgyrogroup $K$ of $G$ such that the quotient space $G/K$ is metrizable. Note that the space $G/K$ is Lindel\"{o}f as a continuous image of the Lindel\"{o}f space $G$. Hence, $G/K$ is separable because of the equivalence between the properties of Lindel\"{o}f and Separable in a metrizable space.

Finally, the compact $L$-subgyrogroup $K$ is separable by Lemma \ref{kyl1}. Therefore, the separability of $G$ just follows from Theorem \ref{kdl1}.
\end{proof}

Since every metrizable strongly topological gyrogroup is feathered, it is clear that we can deduce the following corollaries from Theorem \ref{kdl2}.

\begin{corollary}\label{ktl2}
If a metrizable strongly topological gyrogroup $G$ is isomorphic to a subgyrogroup of a separable strongly topological gyrogroup, we have that $G$ is separable.
\end{corollary}

\begin{corollary}\cite{V,LS}
If a metrizable group $G$ is isomorphic to a subgroup of a separable topological group, then $G$ is separable.
\end{corollary}

Indeed, the conclusion of Corollary \ref{ktl2} remains valid if $G$ is a subgyrogroup of a strongly topological gyrogroup $X$ with countable cellularity.

\begin{lemma}\label{4dl1}\cite{BL3}
Let $(G,\tau ,\oplus)$ be a strongly topological gyrogroup with a symmetric open neighborhood base $\mathscr U$ at $0$. If $c(G)\leq \kappa$, then $G$ is left $\kappa$-bounded.
\end{lemma}

\begin{corollary}
If a metrizable strongly topological gyrogroup $G$ is isomorphic to a subgyrogroup of a strongly topological gyrogroup with countable cellularity, we have that $G$ is separable.
\end{corollary}

\begin{proof}
Assume that $X$ is a strongly topological gyrogroup with countable cellularity. It follows from Lemma \ref{4dl1} that $X$ is left $\omega$-narrow. Then, $G$ is left $\omega$-narrow by Proposition \ref{kmt1}. Since $G$ is first countable and it follows from Theorem \ref{kdl3} that $G$ has a countable base. Therefore, $G$ is separable.
\end{proof}

Moreover, the Theorem \ref{kdl2} is also valid when $G$ is a locally compact strongly topological gyrogroup. It follows from \cite[3.1~E(b) and 3.3~H(a)]{E} that every locally compact topological gyrogroup is feathered. Therefore, we have the following results.

\begin{corollary}
If a locally compact strongly topological gyrogroup $G$ is isomorphic to a subgyrogroup of a separable strongly topological gyrogroup, we have that $G$ is separable.
\end{corollary}

\begin{corollary}\cite{CI}
If a locally compact topological group $G$ is isomorphic to a subgroup of a separable topological group, then $G$ is separable.
\end{corollary}

It follows from \cite[Proposition 4.3.36]{AA} that every closed subspace of a feathered space is feathered. Moreover, the class of feathered topological gyrogroups is closed under countable products by Theorem \ref{dlc2}. Therefore, we obtain the following corollaries.

\begin{corollary}
Let $G$ be a separable locally compact strongly topological gyrogroup and $H$ be a separable feathered strongly topological gyrogroup. If a strongly topological gyrogroup $F$ is isomorphic to a closed subgyrogroup of $G\times H$, then $F$ is separable.
\end{corollary}

\begin{corollary}
Let $G$ be a separable metrizable strongly topological gyrogroup and $H$ be a separable feathered strongly topological gyrogroup. If a strongly topological gyrogroup $F$ is isomorphic to a closed subgyrogroup of $G\times H$, then $F$ is separable.
\end{corollary}

It is well-known that every strongly topological gyrogroup is a topological gyrogroup. So, it is natural to pose the following questions.

\begin{question}
If a topological gyrogroup is feathered, is it paracompact?
\end{question}

\begin{question}
If a topological gyrogroup is feathered, is it a $D$-space?
\end{question}

\begin{question}
If a feathered topological gyrogroup $G$ is isomorphic to a subgyrogroup of a separable topological gyrogroup, is $G$ separable?
\end{question}

\end{document}